\documentclass{article}
\usepackage[utf8]{inputenc}
\usepackage{amssymb}
\usepackage{amsmath}
\usepackage{amsfonts}
\usepackage{amsthm}
\usepackage{biblatex}
\usepackage{xcolor}
\usepackage{amsxtra}
\usepackage{amscd}
\usepackage{eucal}
\usepackage[all]{xy}
\usepackage{graphicx}
\usepackage{pifont}
\usepackage{comment}
\usepackage{tikz-cd}
\usepackage{enumitem}
\usepackage{euler}
\usepackage{latexsym}
\usepackage{framed}
\usepackage{fullpage}
\usepackage{hyperref}
\usepackage{kbordermatrix}
\allowdisplaybreaks
\newtheorem{theorem}{Theorem}[section]

\newcommand{\nc}{\newcommand}

\nc{\on}{\operatorname}
\nc{\Hom}{\on{Hom}}
\nc{\Ind}{\on{Ind}}
\nc{\Irr}{\on{Irr}}
\nc{\Res}{\on{Res}}

\title{Convolution Algebras for Finite Reductive Monoids}
\author{Jared Marx-Kuo, \and Vaughan McDonald, \and John M. O'Brien, \and Alexander Vetter}

\newcommand{\Addresses}{{
  \bigskip

  \noindent Jared Marx-Kuo, \textsc{Department of Mathematics, The University of Chicago, Chicago, IL 60637}\par\nopagebreak
  \textit{E-mail address}: \texttt{jmarxkuo@uchicago.edu}

  \medskip

  \noindent Vaughan McDonald, \textsc{Department of Mathematics, Harvard University, Cambridge, MA 02138}\par\nopagebreak
  \textit{E-mail address}: \texttt{vmcdonald@college.harvard.edu}

  \medskip

  \noindent John M. O'Brien, \textsc{Department of Mathematics, Kansas State University, Manhattan, KS 66506}\par\nopagebreak
  \textit{E-mail address}: \texttt{colbyjobrien@ksu.edu}
  
  \medskip
  
  \noindent Alexander Vetter, \textsc{Department of Mathematics and Statistics, Villanova University, Villanova, PA 19085}\par\nopagebreak
  \textit{E-mail address}: \texttt{avetter@villanova.edu}

}}

\date{\today}

\begin{document}

\maketitle

\begin{abstract}
For an arbitrary finite monoid $M$ and subgroup $K$ of the unit group of $M$, we prove that there is a bijection between irreducible representations of $M$ with nontrivial $K$-fixed space and irreducible representations of $\mathcal{H}_K$, the convolution algebra of $K\times K$-invariant functions from $M$ to $F$, where $F$ is a field of characteristic not dividing $|K|$.  When $M$ is reductive and $K = B$ is a Borel subgroup of the group of units, this indirectly provides a connection between irreducible representations of $M$ and those of $F[R]$, where $R$ is the Renner monoid of $M$.  We conclude with a quick proof of Frobenius Reciprocity for monoids for reference in future papers.
    
\end{abstract}

\section{Introduction}
\subsection{Motivation}
Let $M$ be a reductive monoid over a finite field.  Let $G(M)$ be the unit group of $M$, a connected reductive group, with maximal torus $T$ contained in Borel subgroup $B$.  Recall that $M$ has the Renner decomposition $M = \bigsqcup_{r\in R} B\underline{r} B,$ where $R$, the Renner monoid of $M$, plays the role of the Weyl group of a connected reductive group.  It is well-known that $\mathcal{H}(M,B)$, the $B\times B$-invariant convolution algebra of functions from $M$ to $\mathbb{C}$, is isomorphic to the monoid algebra $\mathbb{C}[R]$ of $R$, just as the equivalent convolution algebra $\mathcal{H}(G,B)$ of a connected reductive group is isomorphic to the group algebra of the Weyl group.  

In the group case, the Borel-Matsumoto theorem implies a bijection between the irreducible representations $(\pi, V)$ of $G$ with nonzero Borel-fixed space $V^B = \{v\in V: \pi(b)v = v\;\forall b\in B\}$ and irreducible representations of $\mathcal{H}(G,B)$. Since the representation theory of $\mathcal{H}(G,B)$ reflects the representation theory of the Weyl group of $G$, the Borel-Matsumoto Theorem classifies many irreducible representations of $G$.  

In this paper, we prove an analogous result to the Borel-Matsumoto theorem for finite monoids.  We prove that, for $K$ a subgroup of $G(M)$, there is a bijection between irreducible representations of $M$ with nonzero $K$-fixed subspaces and representations of the convolution algebra of $K\times K$-invariant functions from $M$ to $F$, where $F$ is of characteristic not dividing $|K|$.  When $M$ is reductive, $K = B$, and $F = \mathbb{C}$, we get the desired connection between representation theory of $M$ and that of its Renner monoid via $\mathcal{H}(M,B)$.

We hope to extend the result to the case of $p$-adic reductive monoids.  For $p$-adic reductive groups, a nearly identical proof replacing summation with integration with respect to a Haar measure works.  However, subtleties related to the nature of smooth representations of monoids prevented a direct extension of the proof from that of finite monoids.  In a future paper, we hope to find an alternative proof.

\section{A Borel-Matsumoto Theorem for Finite Monoids}
\newtheorem{lemma}{Lemma}
Let $M$ be a finite monoid, $G(M)$ the group of units of $M$, $K$ a subgroup of $G(M)$, and $F$ a field of characteristic not dividing $|K|$. 

For $\phi, \psi: M\rightarrow F$, Godelle \cite{godelle} defines the convolution product $\phi * \psi$ by 
\begin{equation*}
    (\phi*\psi)(m) = \sum\limits_{yz=m}\phi(y)\psi(z)
\end{equation*}

Similarly, for $(\pi, V)$ a representation of M and $\phi$ as above define $\pi(\phi)$ by

\begin{equation*}
    \pi(\phi)v = \sum\limits_{x\in M}\phi(x)\pi(x)v
\end{equation*}


\newtheorem{proposition} {Proposition}
\begin{proposition}
For $\phi, \psi\in \mathcal{H}$, $\pi(\phi*\psi) = \pi(\phi)\circ\pi(\psi).$
\end{proposition}

\begin{proof}
Consider $\pi(\phi)\circ\pi(\psi).$  We have the following:

\begin{align*}
    (\pi(\phi)\circ\phi(\psi))v &= \sum\limits_{x\in M}\phi(x)\pi(x)\sum\limits_{y\in M}\psi(y)\pi(y)v) \\
    &=\sum\limits_{x,y\in M} \phi(x)\psi(y)\pi(x)\pi(y)v \\
    &=\sum\limits_{x,y\in M} \phi(x)\psi(y)\pi(xy)v \\
    &=\sum\limits_{z\in M}\sum\limits_{xy = z} \phi(x)\psi(y)\pi(z)v \\
    &=\sum\limits_{z\in M}(\phi*\psi)(z)\pi(z)v \\
    &= \pi(\phi*\psi)v
\end{align*}

Thus $\pi(\phi)\circ\pi(\psi) = \pi(\phi*\psi)$.
\end{proof}

Let $\mathcal{H}$ be the $F$-algebra of functions from $M$ to $F$ under addition and convolution.  Define, for $v\in V$, 

\begin{equation*}
    \mathcal{H}v=\{\pi(\phi)v:  \phi\in\mathcal{H}\}.
\end{equation*}

Define an action of $M$ on $\mathcal{H}v$ by $m \cdot (\pi(\phi)v) = \pi(m)\pi(\phi)v$.  Define $f_m: M\rightarrow F$ by $f_m(m) = 1, f_m(x) = 0$ for $x\neq m$.  Since $\pi(f_m)v = \pi(m)v$, then $\mathcal{H}v$ is closed under action by $M$.  Thus, it is a subrepresentation. 

Similarly, let $\mathcal{H}_K$ be the $F$-algebra of functions from $M$ to $F$ under convolution that are constant on double-cosets of K; i.e. $\phi: M\rightarrow F$ such that $\phi(m)=\phi(k_1mk_2)$ for all $k_1, k_2 \in K$.  Furthermore, let $V^K = \{v \in V \; | \; \pi(k) v = v \quad \forall  k \in K\}$.

\begin{theorem}
Let $(\pi,V)$ be an irreducible representation of M with $V^K\neq\{0\}$. Then $V^K$ is irreducible as an $\mathcal{H}_K$-module.
\end{theorem}

\begin{proof}
We follow Bump's proof of the group case closely \cite{bump}.  We claim that, for all nonzero $u\in V^K$, that $\mathcal{H}_Ku$:= \{$\pi(\phi)u: \phi \in \mathcal{H}_K$\} equals $V^K$.  In other words, we wish to show that, for all $v\in V^K$ there exists $\phi\in\mathcal{H}_K$ such that $\pi(\phi)u = v$.

Since $(\pi,V)$ is an irreducible representation of M, there are no proper non-trivial subrepresentations in V.  Because there is an M-action on $\mathcal{H}u\neq\{0\}$, then $\mathcal{H}u = V$.  Thus there exists $\psi\in\mathcal{H}$ such that $\pi(\psi)u = v$.

Define $\phi\in\mathcal{H}$ by, for $x\in M$ 
\begin{equation*}
    \phi(x) = \frac{1}{|K|^2}\sum\limits_{k_1,k_2\in K}\psi(k_1xk_2)
\end{equation*}
Since $\phi$ must be invariant over left and right cosets of K, $\phi$ lies in $\mathcal{H}_K$.  Now consider the following:

\begin{align*}
    \pi(\phi)u = \frac{1}{|K|^2}\sum\limits_{k_1,k_2\in K}\sum\limits_{x\in M}\psi(k_1xk_1)\pi(x)u
\end{align*}
Notice that $x\mapsto k_1^{-1}xk_2^{-1}$ is a bijection from M to M, as it has an inverse $x\mapsto k_1xk_2$.  Thus we can make the following change of variables:

\begin{align*}
    \pi(\phi)u &= \frac{1}{|K|^2}\sum\limits_{k_1,k_2\in K}\sum\limits_{x\in M}\psi(x)\pi(k_1^{-1}xk_2^{-1})u \\
    &= \frac{1}{|K|^2}\sum\limits_{k_1,k_2\in K}\sum\limits_{x\in M}\psi(x)\pi(k_1)^{-1}\pi(x)\pi(k_2)^{-1}u.
\end{align*}
Since $u\in V^K$, we have that $\pi(k_2)^{-1}u = u$.  Thus,

\begin{align*}
    \pi(\phi)u &= \frac{1}{|K|}\sum\limits_{k_1\in K}\sum\limits_{x\in M}\psi(x)\pi(k_1)^{-1}\pi(x)u \\
    &= \frac{1}{|K|}\sum\limits_{k_1\in K}\pi(k_1)^{-1}\sum\limits_{x\in M}\psi(x)\pi(x)u \\
    &=\frac{1}{|K|}\sum\limits_{k_1\in K}\pi(k_1)^{-1}\pi(\psi)u.
\end{align*}
Since $\pi(\psi)u = v$ and $v\in V^K$,

\begin{align*}
    \pi(\phi)u &= \frac{1}{|K|}\sum\limits_{k_1\in K}\pi(k_1)^{-1}v = v.
\end{align*}

Thus, for all $v\in V^K$ there exists $\phi\in\mathcal{H}_K$ such that $\pi(\phi)u = v$.  Thus, $V^K$ is irreducible as an $\mathcal{H}_K$-module.
\end{proof}

Denote, for $(\pi,V)$ a representation of M, let $(\pi|_G,V)$ be the restricted representation of $G(M)$ defined by $\pi|_G(g)=\pi(g)$ for $g\in G(M)$.  Define the contragredient representation of $G(M)$ $(\hat{\pi}|_G, \hat{V})$ by $\langle\pi|_G(g)v, \hat{v}\rangle = \langle v, \hat{\pi}|_G(g^{-1})\hat{v}\rangle$ for all $g\in G(M)$. 

\begin{lemma}
Let l: $V^K \rightarrow F$ be a linear functional.  Then there exists $\hat{v}\in\hat{V}^K$ such that for all $v\in V^K$, l(v) = $\langle v,\hat{v}\rangle$. \cite{bump}
\end{lemma}

\begin{proof}
    Let $\hat{v}_0$ be a linear functional on V that restricts to l on $V^K$.
    
    Define $\hat{v} = \frac{1}{|K|}\sum_{k\in K}\hat{\pi}|_G(k)\hat{v}_0$.  For $v\in V^K$, then, we have the following equalities:
    
    \begin{align*}
        \langle v, \hat{v}\rangle &= \frac{1}{|K|}\sum\limits_{k\in K}\langle v, \hat{\pi}|_G(k)\hat{v}_0\rangle \\ 
        &=\frac{1}{|K|}\sum\limits_{k\in K}\langle\pi|_G(k)^{-1}v, \hat{v}_0\rangle \\
        &= \frac{1}{|K|}\sum\limits_{k\in K}\langle\pi(k)^{-1}v, \hat{v}_0\rangle \\
        &= \frac{1}{|K|}\sum\limits_{k\in K}\langle v, \hat{v}_0\rangle \\
        &= l(v)
    \end{align*}
\end{proof}
\begin{lemma}
If $V^K\neq 0$ then $\hat{V}^K\neq 0$. \cite{bump}
\end{lemma}

\begin{lemma}
Let $R$ be an algebra over $F$, and $N_1, N_2$ simple $R$-modules that are finite-dimensional as vector spaces over F.  If there exist linear functionals $L_i: N_i \rightarrow F$ and $n_i\in N_i$ such that $L_i(n_i)\neq 0$ and $L_1(rn_1) = L_2(rn_2)$ for all $r\in R$, then $N_1\cong N_2$ as R-modules. \cite{bump}
\end{lemma}
We particularly care about the case when two representations $(\pi_i,V_i)$ share matrix coefficients $\langle \pi_i(m)v, \hat{v}_0 \rangle$ for all $m\in M$.  

\begin{lemma}
Let $(\pi,V)$ and $(\sigma, W)$ be two irreducible representations of $M$ with nonzero matrix coefficients $\langle\pi(m)v, \hat{v}_0\rangle = \langle\sigma(m)w, \hat{w}_0\rangle$ for some v, $v_0$, w, $w_0$, and all $m\in M$.  Then $(\pi,V)\cong(\sigma,W)$.
\end{lemma}

\begin{proof}
Define actions of $F[M]$ on $V$ and $W$ by letting $mv = \pi(m)v$ and $mw = \sigma(m)w$ for all $v\in V, w\in W,$ and $m\in M$ respectively and then extending by linearity.  Thus $V$ and $W$ become $F[M]$-modules. Because the representations are each irreducible, $V$ and $W$ are simple as $F[M]$-modules. Since $\langle mv, \hat{v}_0\rangle = \langle mw, \hat{w}_0\rangle$ for all $m\in M$ are two equal linear functionals on $V$ and $W$, then $V\cong W$ as $F[M]$-modules by Lemma 4.  Equivalently, $(\pi,V)\cong(\sigma,W)$.
\end{proof}

Now we prove the second half of the Borel-Matsumoto Theorem.

\begin{theorem}
If $(\pi, V)$ and $(\sigma, W)$ are two irreducible representations of $M$ with $V^K$ and $W^K$ nonzero and isomorphic as $\mathcal{H}_K$-modules, then $(\pi,V)\cong(\sigma,W)$.
\end{theorem}
\begin{proof}
Let $\lambda: V^K\rightarrow W^K$ be an isomorphism of $\mathcal{H}_K$-modules and $l: W^K\rightarrow F$ be a linear functional not equal to zero.  Then there exist $\hat{v}\in\hat{V}^K$ and $\hat{w}\in\hat{W}^K$ such that $(l\circ\lambda)(v) = \langle v, \hat{v}\rangle$ and $l(w) = \langle w, \hat{w}\rangle$ for all $v\in\hat{V}^K, w\in\hat{W}^K$.  Furthermore, there exist $w_0\in W^K, v_0\in V^K$ such that $\langle w_0, w\rangle\neq 0$ since $l$ is nontrivial and $v_0 = \lambda^{-1}(w_0)$ since $\lambda$ is an isomorphism.

Then for $\phi\in\mathcal{H}_K$, we have that

\begin{equation}\label{eq1}
    \langle\sigma(\phi)w_0,\hat{w}\rangle = \langle\sigma(\phi)\lambda(v_0), \hat{w}\rangle = \langle\lambda(\pi(\phi)v_0), \hat{w}\rangle = (l\circ\lambda)(\pi(\phi)v_0) = \langle\pi(\phi)v_0, \hat{v}\rangle.
\end{equation}

We show that equation \ref{eq1} holds for all $\phi\in\mathcal{H}$ as well as $\mathcal{H}_K$.  For $\phi\in\mathcal{H}$, define $\phi_K\in\mathcal{H}_K$ by 

\begin{equation*}
    \phi_K(x) = \frac{1}{|K|^2}\sum\limits_{k_1, k_2\in K}\phi(k_1xk_2)
\end{equation*}
for all $x\in M$.  By equation \ref{eq1}, then $\langle\pi(\phi_K)v_0, \hat{v}\rangle = \langle\sigma(\phi_K)w_0, \hat{w}\rangle$. Furthermore, we have that 

\begin{align*}
    \langle \pi(\phi_K)v_0, \hat{v}\rangle &= \langle\frac{1}{|K|^2}\sum\limits_{k_1, k_2\in K}\sum\limits_{x\in M}\phi(k_1xk_2)\pi(x)v_0, \hat{v}\rangle \\
    &= \frac{1}{|K|^2}\langle\sum\limits_{k_1, k_2\in K}\sum\limits_{x\in M}\phi(x)\pi(k_1)^{-1}\pi(x)\pi(k_2)^{-1}v_0,  \hat{v}\rangle \\
    &= \frac{1}{|K|^2}\langle\sum\limits_{k_1, k_2\in K}\pi(k_1)^{-1}\circ(\sum\limits_{x\in M}\phi(x)\pi(x))\circ\pi(k_2)^{-1}v_0, \hat{v}\rangle \\
    &=\frac{1}{|K|^2}\langle\sum\limits_{k_1, k_2\in K}\pi(k_1)^{-1}\pi(\phi)\pi(k_2)^{-1}v_0, \hat{v}\rangle \\
    &=\frac{1}{|K|^2}\sum\limits_{k_1, k_2\in K}\langle\pi(k_1)^{-1}\pi(\phi)\pi(k_2)^{-1}v_0, \hat{v}\rangle \\
    &=\frac{1}{|K|^2}\sum\limits_{k_1, k_2\in K}\langle\pi|_G(k_1)^{-1}\pi(\phi)\pi|_G(k_2)^{-1}v_0, \hat{v}\rangle \\
    &= \frac{1}{|K|^2}\sum\limits_{k_1, k_2\in K}\langle\pi(\phi)\pi|_G(k_2)^{-1}v_0, \hat{\pi}|_G(k_1)\hat{v}\rangle.\\
    &= \frac{1}{|K|^2} \sum_{k_1, k_2 \in K} \langle\pi(\phi)v_0, \hat{v}\rangle \\
    &= \langle\pi(\phi)v_0, \hat{v}\rangle.
\end{align*}
since $v_0\in V^K$ and $\hat{v}\in\hat{V}^K$.

Thus $\langle\pi(\phi_K)v_0, \hat{v}\rangle = \langle\pi(\phi)v_0, \hat{v}\rangle$ for all $\phi\in\mathcal{H}$.  Similarly, $\langle\sigma(\phi_K)w_0, \hat{w}\rangle = \langle\sigma(\phi)w_0, \hat{w}\rangle$.  
With this information, then, we have that $\langle\pi(\phi_K)v_0,\hat{v}\rangle = \langle\sigma(\phi_K)w_0,\hat{w}\rangle$ implies that $\langle\pi(\phi)v_0,\hat{v}\rangle = \langle\sigma(\phi)w_0, \hat{w}\rangle$.

Let $\phi_m\in\mathcal{H}$ for all $m\in M$ be the function that sends all x in M with $x\neq m$ to 0 and m to 1.  Then $\pi(\phi_m)v = \pi(m)v$ and $\sigma(\phi_m)w = \sigma(m)w$.

Thus, we have that $\langle\pi(m)v_0, \hat{v}\rangle = \langle\sigma(m)w_0, \hat{w}\rangle$ for all $m\in M$.  By Lemma 5, then, $(\pi,V)$ and $(\sigma,W)$ are equivalent.
\end{proof}

\section{Frobenius Reciprocity}
 Although Doty alluded to the fact that Frobenius Reciprocity holds for monoids \cite{doty}, he left it without proof.   For completeness, we give an explicit proof.

Let $M$ be a finite monoid, $G(M)$ its group of units, $N$ a submonoid of $M$, $G(N)$ its group of units, and $(\pi, V)$ a representation of $M$.  Define the vector space $Ind_N^M V$ as follows:

\begin{equation*}
\Ind_N^{M} V = \{f : M \rightarrow V \; | \; f(nm) = \pi(n)f(m)\, \quad \forall n\in N, \; m\in M\} 
\end{equation*}
Define $(\pi^M, Ind_N^M V)$ by $\pi^M(m)f(x) = f(xm)$ for all m.  

\begin{lemma}
The pair $(\pi^M, Ind_N^M V)$ is a representation of M.
\end{lemma}
\begin{proof}
First, we check that $Ind_N^M V$ is closed under the action of $\pi^M(m)$. Trivially, if $f(nx)=\pi(n)f(x)$ then $\pi^M(m)f(nx) = f(nxm) = \pi(n)f(xm)$ for all $m\in M, n\in N$.  

We check that $\pi^M(m)$ is linear for all $m$.
$$\forall z \in F, \; \forall f,g \in Ind_N^M  \quad z \pi^M(m) f(x) = z f(xm) = \pi^M(m) (zf)(x) $$
$$\pi^M(m)(f + g)(x) = (f + g)(xm) = \pi^M(m)(f)(x) + \pi^M(m)(g)(x) $$

Now, we check that $\pi^M$ is a homomorphism of monoids.  Let $m, x, y\in M$.  Then $\pi^M(mx)f(y) = f(ymx) = \pi^M(x)f(ym) = \pi^M(m)\pi^M(x)f(y)$.  Finally, $\pi^M(1)f(x) = f(x)$, implying that $\pi^M$ maps the identity to the identity.  Clearly, then, $\pi^M(mx)=\pi^M(m)\pi^M(x)$, and $(\pi^M, Ind_N^M)$ is a representation of M.
\end{proof}
 Thus we can call $(\pi^M, Ind_N^M V)$ the induced representation of M.  We have that 

\begin{theorem}
If $(\pi, V)$ is a representation of N, a submonoid of M, and $(\sigma, W)$ a representation of M, then $\Hom_M(W,Ind_N^M V)\cong \Hom_N(W,V)$ as vector spaces.

\end{theorem}
\begin{proof}
For $\phi\in \Hom_M(W,Ind_N^M V)$, define $F: \Hom_M(W,Ind_N^M V)\rightarrow \Hom_N(W,V)$ by $F(\phi)$, such that $F(\phi)(w) = \phi(w)(1)$, 1 being the identity element of $M$.  We first show that $F(\phi)$ is linear.  Because $\phi$ is linear, 
$$F(\phi)(w + w_0) = \phi(w + w_0)(1) = \phi(w)(1) + \phi(w_0)(1) = F(\phi)(w) + F(\phi)(w_0)$$
and for $z\in F$,
$$F(\phi)(zw) = \phi(zw)(1) = z\phi(w)(1) = zF(\phi)(w)$$
We now claim that $F(\phi)$ is a morphism of $N$-modules For $n\in N$, 
\begin{align*} 
F(\phi)(\sigma(n)w) & = \phi(\sigma(n)w)(1) = \pi^M(n)\phi(\sigma(1)w)(1) \\
& = \phi(w)(n) = \pi(n)\phi(w)(1) = \pi(n)F(\phi)(w)
\end{align*}
Thus $F(\phi)$ is an $N$-module homomorphism from $W$ to $V$.  Since 
$$F(\phi + \psi) (w) = (\phi + \psi)(w)(1) = \phi(w)(1) +\psi(w)(1) = F(\phi)(w) + F(\psi)(w)$$
and $F(z \cdot \phi)(w) = (z\phi)(w)(1) = zF(\phi)(w)$, then F is a vector space homomorphism. For $\tau\in \Hom_N(W,V)$, let $G: \Hom_N(W,V)\rightarrow \Hom_M(W,Ind_N^M V)$ such that 
$$(G(\tau)(w))(m) = G(\tau)(w)(m) =  \tau(\sigma(m)w)$$
then, $\tau(\sigma(nm)w) = \tau(\sigma(n)\sigma(m)w) = \pi(n)\tau(\sigma(m)w)$, so $G(\tau)(w)$ is in $Ind_N^M$. We check that $G(\tau)(-)(m)$ is linear.  This follows from the definition: \newline
\begin{align*} 
G(\tau)(w + w_0)(m) 
& = \tau(\sigma(m)(w + w_0)) = \tau(\sigma(m)w + \sigma(m)w_0) \\
& = \tau(\sigma(m)w) + \tau(\sigma(m)w_0) = G(\tau)(w)(m) + G(\tau)(w_0)(m)
\end{align*} 
and for $z\in F$, we have
\[
G(\tau)(zw)(m) = \tau(\sigma(m)(zw)) = z\tau(\sigma(m)(w))
\]
Next, we check that $G(\tau)$ respects M.  We have that for $x\in M$, 
\begin{align*} 
G(\tau)(\sigma(x)w)(m) 
& = \tau(\sigma(m)\sigma(x)w) = \tau(\sigma(mx)w) \\
& = (\pi^M(x)\circ\tau)(\sigma(m)w) = \pi^M(x)(G(\tau)(w)(m))
\end{align*}
Thus $G(\tau)\in \Hom_M(W, Ind_N^M V)$.  Finally, we check that G itself is linear:
\begin{align*} 
G(\tau + \eta)(w)(m) 
& = (\tau + \eta)(\sigma(m)w) \\ 
& = \tau(\sigma(m)w) + \eta(\sigma(m)w) = G(\tau)(w)(m) + G(\eta)(w)(m)
\end{align*}
and for $k\in K, \quad  G(k\tau)(w)(m) = k(\tau(\sigma(w)m)) = k \cdot G(\tau)(w)(m)$.  Thus G is a homomorphism of vector spaces.

Now, we show that F and G are inverses.  First, we check the mapping $G\circ F: Hom_M(W,Ind_N^M)\rightarrow Hom_M(W,Ind_N^M)$.  Let $\phi\in Hom_M(W,Ind_N^M)$.  Then $G\circ F(\phi)$ works as follows.  Since $F(\phi)$ is the map sending w to $\phi(w)(1)$, 
\begin{align*} 
(G\circ F)(\phi)(w)(m) 
& = G(F(\phi))(w)(m) = F(\phi)(\sigma(m)w) \\
& = F(\phi)(\sigma(1*m)w) = \pi^M(m)F(\phi)(w) \\
& = \pi^M(m)\phi(w)(1) = \phi(w)(m)
\end{align*}
by definition of the induced representation.  Since $(G\circ F)(\phi)(w)(m) = \phi(w)(m)$, $G\circ F$ is the identity morphism on $Hom_M(W,Ind_N^W)$.

Next, we check $F\circ G: Hom_N(V,W)\rightarrow Hom_N(V,W)$.  Let $\tau\in Hom_N(V,W)$.  Then
\begin{align*} 
(F \circ G)(\tau)(w)(n) & = F(G(\tau))(w)(n) \\
& = G(\tau)(w)(1 \cdot n) = \tau(\pi(n)w) = \sigma(n) \tau(w) = \tau(w)(n)
\end{align*}
Thus $F\circ G$ is the identity morphism on $Hom_N(V, W)$.  Since we have that both $G\circ F$ and $F\circ G$ are identity morphisms on their respective domains, they are inverses.  Thus, we have that $Hom_M(W, Ind_N^M)\cong Hom_N(W,V)$ as vector spaces over $F$.
\end{proof}

\section{Further directions}
In a possible sequel, we would like to study smooth representations of $p$-adic reductive monoids.  In the group case, the Borel-Matsumoto theorem extends easily to smooth representations of $p$-adic reductive groups \cite{bump}.  The proof is virtually identical, with summation replaced by integration over a Haar measure.  A similar result may hold for $p$-adic reductive monoids; however, the authors ran into some difficulty defining a suitable measure.  Several subtle differences between the properties of smooth representations of $p$-adic reductive monoids and those of $p$-adic reductive groups prevented an immediate extension of the proof for the group case. A better description of smooth representations of $p$-adic reductive monoids may enable an alternative proof.

Also, for a finite reductive monoid $M$ with Borel subgroup $B$, we would like to explore reconstructing the irreducible representations of $M$ with nonzero $B$-fixed space from those of $\mathcal{H}(M,B)$.  The Borel-Matsumoto theorem guarantees the existence of a bijection between irreducible representations of $M$ with nonzero $B$-fixed space and irreducible representations of $\mathcal{H}(M,B)$; however, it does not explicitly construct the bijection.  In the finite reductive group case, Deligne and Lusztig used $\ell$-adic cohomology of certain varieties associated with $G$ to construct irreducible representations of $G$ \cite{deligne-lusztig}.  We believe that a similar technique could work in the monoid case.

\section{Acknowledgements}
This research was conducted at the 2018 University of Minnesota-Twin Cities REU in Algebraic Combinatorics. Our research was supported by NSF RTG grant DMS-1745638. We would like to thank Benjamin Brubaker and Andy Hardt for their help and support during our time in Minnesota.

\nocite{*}
\printbibliography

\Addresses

\end{document}